\documentclass[english]{article}
\usepackage[latin9]{inputenc}
\usepackage{geometry}
\usepackage{color}
\usepackage{babel}
\usepackage{amsmath}
\usepackage{amsthm}
\usepackage{amssymb}
\usepackage[numbers]{natbib}
\usepackage[all]{xy}
\usepackage{breakurl}

\def \RR{{\mathbb R}}
\def \ZZ{{\mathbb Z}}

\makeatletter
\numberwithin{equation}{section}
\numberwithin{figure}{section}
\theoremstyle{plain}
\newtheorem{thm}{\protect\theoremname}
  \theoremstyle{plain}
  \newtheorem{prop}[thm]{\protect\propositionname}
  \theoremstyle{plain}
  \newtheorem{lem}[thm]{\protect\lemmaname}
  \newtheorem{cor}[thm]{\protect\corollaryname}

\usepackage[T1]{fontenc}
\usepackage{ae,aecompl}
\usepackage[T1]{fontenc}
\usepackage{amsmath}
\usepackage{graphicx}

\usepackage{tikz}

\def \vski{\vspace{12pt}}

\makeatother

\providecommand{\corollaryname}{Corollary}
  \providecommand{\lemmaname}{Lemma}
  \providecommand{\propositionname}{Proposition}
\providecommand{\theoremname}{Theorem}

\begin{document}

\title{\textbf{Remarks on the speeds of a class of random walks on the integers}}

\author{Maher Boudabra, Greg Markowsky\\
Monash University}
\maketitle
\begin{abstract}
In recent years, there has been an interest in deriving certain important probabilistic results as consequences of deterministic ones; see for instance \cite{beig} and \cite{acc}. In this work, we continue on this path by deducing a well known equivalence between the speed of random walks on the integers and the growth of the size of their ranges. This result is an immediate consequence of the Kesten-Spitzer-Whitman theorem, and by appearances is probabilistic in nature, but we will show that it follows easily from an elementary deterministic result. We also investigate the common property of recurrent random walks of having speed zero, and show by example that this property need not be shared by deterministic sequences. However, if we consider the inter-arrival times (times at which the sequence is equal to 0) then we find a sufficient deterministic condition for a sequence to have zero speed, and show that this can be used to derive several probabilistic results.

\end{abstract}

\vski
\vski

In many probabilistic settings, the {\it speed} of a random walk $X_n$ on the integers $\ZZ$ is defined to be $\lim_{n\to \infty}\frac{|X_n|}{n}$, whenever the limit exists. Over time, a number of interesting equivalences have been derived relating the speed with other quantitative aspects of the walk in question. The standard proofs in the literature are probabilistic, however several of these results are actually deterministic ones in disguise. The purpose of this note is to describe several examples of this phenomenon.

\vski

\section{Equivalence of speeds of the range and of the random walk}

Let us refer to a stochastic process of the form
\[
X_{0}=0,\,\,X_{n}=\sum_{k=1}^{n}Z_{k}
\]
 where the $Z_{i}$'s are i.i.d random variables with $\mathbb{P}(Z_{i}=1)=1-\mathbb{P}(Z_{i}=-1)=p$,
as \emph{simple random walk} on $\mathbb{Z}$. We then have the following
theorem, which is a special case of the Kesten-Spitzer-Whitman Theorem (see \cite[Ch. 1]{spitz})

\begin{thm} \label{oldbusted}
\[
\underset{{\scriptscriptstyle n\rightarrow+\infty}}{\lim}\frac{R_{n}}{n}=\mathbb{P}\left\{ \text{\ensuremath{X_{n}\neq0} }\forall n>0\right\} =\mathbb{P}\{\text{\ensuremath{X_{n}} never returns to }0\}
\]
where
\[
R_{n}=\text{card}\{X_{0},...,X_{n}\}
\]
\end{thm}

In other words, the number of distinct values taken by $X_{n}$ up to time $n$ can be used to calculate the probability of never returning to zero. It can also be shown that (see \cite{norris})

\[
\mathbb{P}\{\text{\ensuremath{X_{n}} never returns to }0\}=|2p-1|={\displaystyle \lim_{{\scriptscriptstyle n\rightarrow+\infty}}}\frac{|X_{n}|}{n},
\]

with the final equality being simply the Law of Large Numbers, and therefore

\begin{equation}
\underset{{\scriptscriptstyle n\rightarrow+\infty}}{\lim}\frac{R_{n}}{n}=\underset{{\scriptscriptstyle n\rightarrow+\infty}}{\lim}\frac{|X_{n}|}{n}\label{eq:limits}
\end{equation}

We consider the left side of this equation to be the speed of the range process $R_n$, and as discussed before the right is the speed of the walk $X_n$. The standard proofs of the preceding facts are all probabilistic; however, despite appearances, the equality (\ref{eq:limits}) is deterministic in nature, as we now show.

Let $x_{n}$ be an (deterministic) integer valued sequence satisfying
\begin{equation}
\begin{vmatrix}x_{n+1}-x_{n}\end{vmatrix}\leq1.\label{increment}
\end{equation}
As before set
\[
r_{n}=\text{card}\left\{ x_{0},x_{1},...,x_{n}\right\}
\]
We give the following result.
\begin{prop}
\label{prop:|ell|}If
\[
\underset{{\scriptscriptstyle n\rightarrow+\infty}}{\lim}\frac{x_{n}}{n}=\ell
\]
 then
\[
\underset{{\scriptscriptstyle n\rightarrow+\infty}}{\lim}\frac{r_{n}}{n}=\begin{vmatrix}\ell\end{vmatrix}.
\]
\end{prop}
\begin{proof}
Without loss of generality we assume $x_{0}=0$, and suppose first
that $\ell>0$ . Let $0<\varepsilon<\ell$, as
\[
\underset{{\scriptscriptstyle n\rightarrow+\infty}}{\lim}\frac{x_{n}}{n}=\ell
\]
 then there is $N>0$ such that
\begin{equation}
n\left(\ell-\varepsilon\right)\leq x_{n}\leq n\left(\ell+\varepsilon\right)\,\,\,\,\text{for \ensuremath{n>N} }\label{eq:inequality}
\end{equation}
Because of (\ref{increment}), we get
\[
\lfloor n\left(\ell-\varepsilon\right)\rfloor\leq r_{n}
\]
Furthermore
\[
\text{card}\{x_{N},...,x_{n}\}\leq n(\varepsilon+\ell)
\]
since all terms are positive and bounded above by $n(\varepsilon+\ell)$.
Therefore
\[
r_{n}\leq r_{N-1}+\text{card}\{x_{N},...,x_{n}\}\leq r_{N-1}+n(\varepsilon+\ell)
\]
Hence
\[
\ell-\varepsilon\leq\liminf_{+\infty}\frac{r_{n}}{n}\leq\limsup_{+\infty}\frac{r_{n}}{n}\leq\ell+\varepsilon
\]
The result now follows. The case $\ell<0$ is similar by considering
$-x_{n}$ instead. Assume now $\ell=0$, in this case we get
\[
|x_{n}|\leq r_{N-1}+2\varepsilon n
\]
instead of (\ref{eq:inequality}) for $n>N$, and the rest follows
as before.
\end{proof}

We now present a probabilistic application of this result. It is of interest to define random walks whose increments are no longer i.i.d, but are instead assumed to be ergodic (see \cite{shiry} for relevant definitions). Let $(\xi_{n})_{n}$ be an ergodic sequence of random variables with
values in $\{-1,0,1\}$, and let $X_{n}=\sum_{k=1}^{n}\xi_{k}$. As before let $R_{n}=\text{card}\{X_{0},...,X_{n}\}$. We then have the following proposition, which generalizes Theorem \ref{oldbusted}, and which is an immediate consequence of Proposition \ref{prop:|ell|}.

\begin{prop}
\[
\underset{{\scriptscriptstyle n\rightarrow+\infty}}{\lim}\frac{R_{n}}{n}=|E(\xi_{0})|
\]
\end{prop}


If we put the weaker condition $\begin{vmatrix}x_{n+1}-x_{n}\end{vmatrix}\leq m$
where $m$ is an integer greater than $1$, then we have the following
generalization
\begin{prop}
\label{inequality lim range}If ${\displaystyle \underset{{\scriptscriptstyle n\rightarrow+\infty}}{\lim}\frac{x_{n}}{n}=\ell}$
then
\begin{equation}
\frac{|\ell|}{m}\leq\liminf_{+\infty}\frac{r_{n}}{n}\leq\limsup_{+\infty}\frac{r_{n}}{n}\leq\min(1,|\ell|)\label{eq:inequality-1}
\end{equation}
\end{prop}

The proof is based on the following technical lemma.

\begin{lem}
\textbf{\textup{(Maximal range inequality)}}

Under the previous assumption we have
\begin{equation}
\frac{{\displaystyle \max_{0\leq k\leq n}\begin{vmatrix}x_{k}-x_{0}\end{vmatrix}}}{m}+1\leq r_{n}\label{maximal inequality}
\end{equation}
\end{lem}
\begin{proof}
We proceed by induction on $n$ and without loss of generality we
may assume $x_{0}=0$ (otherwise replace each $x_n$ by $x_n-x_0$). The property holds for $n=0$ since
\[
\underset{=0}{\underbrace{\frac{{\displaystyle \max_{0\leq k\leq0}\begin{vmatrix}x_{k}\end{vmatrix}}}{m}}}+1\leq1=r_{0}
\]
For $n>0$ set $V_{n}:=\{x_{0},...,x_{n}\}$. Since ${\displaystyle r_{n+1}=r_{n}+\mathbf{1}_{\left\{ x_{n+1}\notin V_{n}\right\} }}$
two situation occur:
\begin{itemize}
\item If $|x_{n+1}|>{\displaystyle \max_{0\leq k\leq n}\begin{vmatrix}x_{k}\end{vmatrix}}$
then $x_{n+1}\notin V_{n}$ and so
\[
\begin{aligned}\max_{0\leq k\leq n+1}\begin{vmatrix}x_{k}\end{vmatrix} & =\begin{vmatrix}x_{n+1}\end{vmatrix}\\
 & \leq\begin{vmatrix}x_{n}\end{vmatrix}+m\\
 & \leq{\displaystyle \max_{0\leq k\leq n}\begin{vmatrix}x_{k}\end{vmatrix}}+m
\end{aligned}
\]
and then
\[
\begin{aligned}r_{n+1} & =r_{n}+1\\
 & \geq\frac{{\displaystyle \max_{0\leq k\leq n}\begin{vmatrix}x_{k}\end{vmatrix}}+2m}{m} & \text{ \ensuremath{{\scriptstyle (induction)}}}\\
 & \geq\frac{{\displaystyle \max_{0\leq k\leq n+1}\begin{vmatrix}x_{k}\end{vmatrix}}}{m}+1
\end{aligned}
\]
\item If $\begin{vmatrix}x_{n+1}\end{vmatrix}\leq{\displaystyle \max_{0\leq k\leq n}\begin{vmatrix}x_{k}\end{vmatrix}}$
then
\[
r_{n+1} \geq r_{n}\geq{\displaystyle \frac{{\displaystyle \max_{0\leq k\leq n}\begin{vmatrix}x_{k}\end{vmatrix}}}{m}}+1={\displaystyle \frac{{\displaystyle \max_{0\leq k\leq n+1}\begin{vmatrix}x_{k}\end{vmatrix}}}{m}}+1
\]
\end{itemize}
Thus the property remains true for $r_{n+1}$, and so (\ref{maximal inequality})
holds for all $n$.
\end{proof}
Now, we are ready to prove Proposition \ref{inequality lim range}. As usual
we assume that $\ell>0$ and $x_{0}=0$, then there exists $N>0$
such that
\[
n(\ell-\varepsilon)\leq x_{n}\leq n(\ell+\varepsilon)\,\,\,\,\text{for \ensuremath{n>N} }
\]
where $0<\varepsilon<\ell$. In particular, by the previous lemma,
we obtain
\[
\frac{\lfloor n(\ell-\varepsilon)\rfloor}{nm}\leq\frac{r_{n}}{n}\leq\frac{\lfloor n(\ell+\varepsilon)\rfloor+r_{N-1}}{n}
\]
where $\lfloor\cdot\rfloor$ denotes the ceiling function. Therefore,
as $\frac{r_{n}}{n}$ does not exceed $1$, we conclude the result.

The case $\ell<0$ is similar by considering
$-x_{n}$ instead. It is straightforward to verify that if $\ell=0$ then $\frac{r_{n}}{n} \to 0$ with an argument analogous to that in Proposition \ref{prop:|ell|}, and the result follows.

\vski

A consequence of Proposition \ref{inequality lim range} is the following.

\begin{prop}
The following statements are equivalent :
\begin{enumerate}
\item $\underset{{\scriptscriptstyle n\rightarrow+\infty}}{\lim}\frac{{\displaystyle \max_{0\leq k\leq n}|x_{n}|}}{n}=0$.
\item $\underset{{\scriptscriptstyle n\rightarrow+\infty}}{\lim}\frac{{\displaystyle x_{n}}}{n}=0$
\item $\underset{{\scriptscriptstyle n\rightarrow+\infty}}{\lim}\frac{{\displaystyle r_{n}}}{n}=0$
\end{enumerate}
\end{prop}
\begin{proof}
It follows from
\[
\xymatrix{1\ar[rr]^{\text{obvious}} &  & 2\\
 & 3\ar@{->}[lu]^{\eqref{maximal inequality}}\ar@{<->}[ru]_{\eqref{eq:inequality-1}}
}
\]
\end{proof}

It should be noted that in higher dimensions the Kesten-Spitzer-Whitman theorem still holds, while our Proposition \ref{prop:|ell|} and ensuing results do not. An easy example proves this: take $x_n$ to be a sequence which essentially fills the space in $\RR^2$, for instance one which winds in a spiral shape around the origin; in this case, if no vertices are repeated, then $r_n =n$ but $\frac{x_n}{n} \to 0$, and $|x_n - x_{n-1}|$ can be taken to be 1. However, we still have a lower bound for ${\displaystyle \liminf_{+\infty}}\frac{r_{n}}{n}$ in terms of $\lim_{n \to \infty}\frac{|x_n|}{n}$. The proof of the following is obtained by the same techniques as for Proposition \ref{inequality lim range}.

\begin{prop}
Let $x_{n}$ be a sequence of $\mathbb{Z}^{d}$ $(d>1)$ such that
$\begin{Vmatrix}x_{n+1}-x_{n}\end{Vmatrix}_{2}\leq m$ for some positive
integer $m$. If $\underset{{\scriptscriptstyle n\rightarrow+\infty}}{\lim}\frac{x_{n}}{n}=\ell$
then
\[
\frac{\begin{Vmatrix}{\scriptstyle \ell}\end{Vmatrix}_{2}}{m}\leq\liminf_{+\infty}\frac{r_{n}}{n}
\]
\end{prop}

\section{The zero-speed case}

Many random walks which are recurrent are known to satisfy $\frac{X_n}{n} \to 0$ a.s. (e.g. simple random walk in $\ZZ$ and $\ZZ^2$), and it is natural to ask whether this should hold for deterministic sequences as well. However, this common occurance really is a probabilistic one, and not deterministic. To be precise, we have the following proposition.

\begin{prop}
\label{thm:For-every-}For every $\ell\in (0,1)$ there exists a sequence
$x:=(x_{n})_{n \geq 0}$ such that

\begin{itemize} \label{}

\item[(i)]  $|x_n-x_{n-1}| \leq 1$ for all $n$.

\item[(ii)] $x_n=0$ for infinitely many $n$.

\item[(iii)] $\underset{{\scriptscriptstyle n\rightarrow+\infty}}{\limsup}\frac{x_{n}}{n} \geq \ell$.

\end{itemize}
\end{prop}

\begin{proof}
We give an explicit construction for $\ell = \frac{1}{2}$, and then indicate how it may be extended to arbitrary $\ell$. Define two sequences $\tau_{n}$ and $t_{n}$ by $\tau_{n} = 2 \cdot 3^n$ and $t_{n} =\frac{\tau_{n}+\tau_{n-1}}{2}$. Now we
define our sequence $x_{n}$ as follows (with $\tau_{-1}=0$)
\[
x_{k}=\begin{cases}
k-\tau_{n-1} & \text{if }\ensuremath{\tau_{n-1} \leq k < t_{n}}\\
\tau_{n}-k & \text{if \ensuremath{t_{n} \leq k < \tau_{n}}}
\end{cases}
\]

Clearly $x_{\tau_n} = 0$, so $x_k = 0$ infinitely often. Note also that $x_{t_n} = \frac{\tau_n - \tau_{n-1}}{2}$, so that $\frac{x_{t_n}}{t_n} = \frac{\tau_n - \tau_{n-1}}{\tau_n + \tau_{n-1}} = \frac{2 \cdot 3^n - 2 \cdot 3^{n-1}}{2 \cdot 3^n + 2 \cdot 3^{n-1}} = \frac{1}{2}$. In order to achieve a larger $\ell < 1$, we let $n_{0}$ be an integer so that

\begin{equation}
\left(\frac{1+\ell}{1-\ell}\right)^{n_{0}}(\frac{2\ell}{1-\ell})>1\label{eq:1}
\end{equation}

Then, we form our sequence in the same manner as before, but with

\[
\begin{alignedat}{1}\tau_{n} & :=2\lfloor\left(\frac{1+\ell}{1-\ell}\right)^{n+n_{0}}\rfloor\\
t_{n} & :=\frac{\tau_{n}+\tau_{n-1}}{2}
\end{alignedat}
\]

It can then be shown that $x_{\tau_n} = 0$, so $x_k = 0$ infinitely often, but we also have $\limsup_{n \to \infty} \frac{x_{t_n}}{t_n} \geq \ell$. Details are more difficult that for the $\ell = \frac{1}{2}$ case, and are omitted.
\end{proof}

On the other hand, there is a sense in which we may interpret the speed of $x_n$ as being deterministic, namely that is controlled by the sequence $(\tau_{k})_{k}$, which we define to be the time of the $k$-th visit to $0$ by the sequence $x$. In particular, we have the following result.

\begin{prop}
If $x\in R$ then
\[
\limsup_{n \to \infty} \frac{|x_n|}{n} \leq\limsup_{k}\frac{1}{2}\left(\frac{\tau_{k}}{\tau_{k-1}}-1\right)
\]
\end{prop}

\begin{proof}
If we denote by $t_{k}$ the time between $\tau_{k-1}$ and $\tau_{k}$
so that $\begin{vmatrix}x_{t_{n}}\end{vmatrix}$ is maximal, then
\[
\begin{vmatrix}x_{t_{k}}\end{vmatrix}\leq\frac{\tau_{k}-\tau_{k-1}}{2}
\]

Now, every $n$ lies inside some $[\tau_{k_{n}-1},\tau_{k_{n}})$
and therefore
\[
\begin{vmatrix}\frac{x_{n}}{n}\end{vmatrix}\leq\begin{vmatrix}\frac{x_{t_{k_{n}}}}{\tau_{k_{n}-1}}\end{vmatrix}\leq\frac{1}{2}\left(\frac{\tau_{k_{n}}}{\tau_{k_{n}-1}}-1\right),
\]

and the result follows.
\end{proof}

\begin{cor}
If a sequence $x\in R$ hits zero at times $\tau_{0}<\tau_{1}<...<\tau_{k}<...$
such that $\frac{\tau_{k}}{\tau_{k-1}}$ converge to $1$, then $x$
has zero speed.
\end{cor}

This shows that if the speed of $x$ is not zero then
necessarily $\frac{\tau_{k}}{\tau_{k-1}}$ does not converge to $1$. As an example, the explicit sequence constructed in Proposition \ref{thm:For-every-} has $\frac{\tau_k}{\tau_{k-1}} = 3$ for all $k$.

\vski

We now present a final application of these ideas. For the purpose of the next proposition, we consider a random walk to be any stochastic process taking values on the integers.

\begin{prop}
\label{prop:Let--be} Let $(X_{n})_{n}$ be any recurrent random walk with
increments taking values in $\{0,\pm1\}$ and which returns to zero at times $\tau_{0}<\tau_{1}<...<\tau_{k}<...$.
If the the sequence $(\tau_{k}-\tau_{k-1})_{k}$ is ergodic, then
$(X_{n})_{n}$ has zero speed.
\end{prop}
\begin{proof}
We have
\begin{equation}
\begin{vmatrix}\frac{X_{n}}{n}\end{vmatrix}\leq\begin{vmatrix}\frac{X_{t_{k_{n}}}}{\tau_{k_{n}-1}}\end{vmatrix}\leq\frac{1}{2}\left(\frac{\tau_{k_{n}}-\tau_{k_{n}-1}}{\tau_{k_n-1}}\right) \leq\frac{1}{2}\left(\frac{\tau_{k_{n}}-\tau_{k_{n}-1}}{\left[{\displaystyle \frac{1}{k_{n}}\times\sum_{j=1}^{k_{n}-1}\tau_{j}-\tau_{j-1}}\right]\times k_{n}}\right)\label{eq:ergodic}
\end{equation}
Now, ${\displaystyle \frac{1}{k_{n}}\times\sum_{j=1}^{k_{n}-1}\tau_{j}-\tau_{j-1}}$
converges a.s. to $E\left(\tau_{1}-\tau_{0}\right)$ which is positive (possibly infinite) by Birkoff's ergodic theorem, and since the distribution of the r.v $\tau_{k_{n}}-\tau_{k_{n}-1}$
does not depend on $n$ then the right hand side in \ref{eq:ergodic} goes to zero as $n$ goes to $\infty$.
\end{proof}

We recover a classical result (\cite[p.8]{solomon}) illustrated by the following corollary.

\begin{cor}
If $(X_{n})_{n}$ is a recurrent Markov chain on $\ZZ$ with increments $0,\pm1$
(eg. a birth-death chain) then its speed is zero.
\end{cor}
\begin{proof}
The inter-return times to zero are independent for a Markov chain.
\end{proof}

\bibliographystyle{IEEEtranN}
\bibliography{bibfile,\string"C:/Users/mbou0007/Google Drive/Maheref\string"}

\end{document}